\documentclass[12pt]{amsart}

\usepackage{aliascnt,amsmath,amssymb,amsthm,amsfonts,mathtools,xcolor,shuffle}
\usepackage[marginparwidth=0pt,margin=24truemm]{geometry}

\definecolor{mylinkcolor}{RGB}{16,156,81}
\definecolor{mycitecolor}{RGB}{20,80,140}

\usepackage{hyperref}
\usepackage[nameinlink]{cleveref}
\hypersetup{
    bookmarksnumbered=true,
    bookmarksopen=true,
    colorlinks=true,
    linkcolor=mylinkcolor,
    citecolor=mycitecolor,
    urlcolor=mycitecolor,
}

\numberwithin{equation}{section}
\allowdisplaybreaks[2]

\theoremstyle{plain}

\crefname{thm}{Theorem}{Theorems}

\newaliascnt{lem}{thm}
\newtheorem{lem}[lem]{Lemma}
\aliascntresetthe{lem}
\crefname{lem}{Lemma}{Lemmas}

\newaliascnt{cor}{thm}

\aliascntresetthe{cor}
\crefname{cor}{Corollary}{Corollaries}

\theoremstyle{remark}
\newaliascnt{rem}{thm}

\aliascntresetthe{rem}
\crefname{rem}{Remark}{Remarks}

\theoremstyle{plain}
\newtheorem{mainthm}{Main Theorem}

\crefname{mainthm}{Main Theorem}{Main Theorems}

\newcommand{\QQ}{\mathbb{Q}}
\newcommand{\ZZ}{\mathbb{Z}}
\newcommand{\HH}{\mathbb{H}}
\newcommand{\summ}[1]{\sum_{\substack{#1}}}
\newcommand{\ds}{\operatorname{ds}}

\title{A Weighted Sum Formula for Double Eisenstein Series}

\author{Henrik Bachmann}
\address{Graduate School of Mathematics, Nagoya University, Nagoya, Japan}
\email{henrik.bachmann@math.nagoya-u.ac.jp}

\date{\today}


\begin{document}

\begin{abstract}
We prove a weighted sum formula for double Eisenstein series.  Its
corresponding identity for the generating series of multiple divisor sums was
conjectured by the author in his master's thesis.  The double Eisenstein series identity
follows from the restricted double-shuffle relations proved by the author and
Tasaka, while the proof of the divisor-sum identity is combinatorial and uses
generating series.
\end{abstract}

\maketitle

\section{Introduction}

Gangl, Kaneko and Zagier introduced double Eisenstein series in
\cite{GKZ} as objects connecting double zeta values and modular forms.  Their
Fourier expansions are built from multiple zeta values and the $q$-series
$g_{k_1,\dots,k_r}$ defined in \eqref{eq:def-g} below.  These $q$-series, which are $q$-analogues
of multiple zeta values, were introduced in the author's master's thesis
\cite{BachmannMSc} and later studied in detail by the author and K\"uhn
\cite{BK}.  We expect that the multiple Eisenstein series $G$ and the
$q$-series $g$ satisfy the same linear relations modulo lower-weight terms.
Motivated by numerical evidence, the author formulated one such identity in
his master's thesis \cite[Vermutung~5.2.7]{BachmannMSc}.  In this note we
prove it in two compatible forms.  \Cref{thm:main-G} gives the homogeneous
weight-$k$ relation between single and double Eisenstein series, and
\Cref{thm:main-g} gives the corresponding identity for the $q$-series $g$,
including its explicit lower-weight correction.

To state the results, we recall the definition of multiple Eisenstein series, which were introduced in \cite{GKZ}.
For $\tau\in\HH=\{z\in\mathbb{C}:\operatorname{Im}(z)>0\}$, define an
order on $\ZZ\tau+\ZZ$ by $\lambda_1\succ\lambda_2$ if and only if
$\lambda_1-\lambda_2\in\HH\cup\mathbb{R}_{>0}$.
For $k_1\geq3$ and $k_2,\dots,k_r\geq2$, define
\begin{align*}
 G_{k_1,\dots,k_r}(\tau)
 \coloneqq
 \summ{\lambda_1,\dots,\lambda_r\in\ZZ\tau+\ZZ\\
 \lambda_1\succ\cdots\succ\lambda_r\succ0}
 \frac{1}{\lambda_1^{k_1}\cdots\lambda_r^{k_r}}.
\end{align*}
For $a\geq3$ and $b\geq2$ with $k=a+b\geq6$, put\footnote{As usual, we set
$\binom{n}{j}=0$ if $j<0$ or $j>n$.}
\begin{align}\label{eq:alpha}
 \alpha_{a,b}
 \coloneqq\frac{60}{b}\binom{k-1}{a-1}^{-1}
 \Biggl((2^{a-2}-1)\binom{k-6}{a-2}
 +(2^{a-3}+2)\binom{k-6}{a-3}
 -\binom{k-6}{a-4}\Biggr).
\end{align}

\begin{mainthm}\label{thm:main-G}
For every integer $k\geq6$, we have
\begin{align}\label{eq:main-G}
 G_k
 =\summ{a+b=k\\a\geq3,\ b\geq2}\alpha_{a,b}\,G_{a,b}.
\end{align}
\end{mainthm}

We next give the corresponding $q$-series statement.  For positive integers
$k_1,\dots,k_r$, set
\begin{align}\label{eq:def-g}
 g_{k_1,\dots,k_r}(q)
 \coloneqq
 \summ{m_1>\cdots>m_r>0\\n_1,\dots,n_r>0}
 \frac{n_1^{k_1-1}}{(k_1-1)!}\cdots
 \frac{n_r^{k_r-1}}{(k_r-1)!}
 q^{m_1n_1+\cdots+m_rn_r},
 \qquad |q|<1.
\end{align}
For the empty index, we use the convention $g_{\varnothing}=1$.
These series can be viewed as $q$-analogues of multiple zeta values, since
for $k_1\geq2$ one has
\begin{align*}
 \lim_{q\to1^-}(1-q)^{k_1+\cdots+k_r}g_{k_1,\dots,k_r}(q)
 =\zeta(k_1,\dots,k_r)
 \coloneqq\summ{m_1>\cdots>m_r>0}
 \frac{1}{m_1^{k_1}\cdots m_r^{k_r}}.
\end{align*}

Let $\mathcal Z$ denote the $\QQ$-algebra of multiple zeta values.  The
Fourier expansion of every multiple Eisenstein series is a
$\mathcal Z[\pi i]$-linear combination of the $q$-series
$g_{k_1,\dots,k_r}$ when setting $q=e^{2\pi i \tau}$ (see \cite{GKZ,BachmannMSc}).  The corresponding relation
among the $q$-series $g$ is the following.

\begin{mainthm}\label{thm:main-g}
For every integer $k\geq6$, we have
\begin{align}\label{eq:main-g}
 g_k
 -\frac{5}{(k-1)(k-2)}g_{k-2}
 +\frac{4}{(k-1)(k-2)(k-3)(k-4)}g_{k-4}
 =\summ{a+b=k\\a\geq3,\ b\geq2}\alpha_{a,b}\,g_{a,b}.
\end{align}
\end{mainthm}

The Fourier expansion in \cite[Theorem~6]{GKZ}, together with
\cite[Corollary on p.~8]{GKZ}, can be used to show that
\Cref{thm:main-G,thm:main-g} are equivalent.  In this note, however, we give
two individual proofs. 

Multiplying \eqref{eq:main-g} by $(1-q)^k$ and letting $q\to1^-$ gives the
weighted sum formula
\begin{align*}
 \zeta(k)
 =\summ{a+b=k\\a\geq3,\ b\geq2}\alpha_{a,b}\zeta(a,b),
\end{align*}
which can also be obtained from the constant term of \Cref{thm:main-G}.

After writing $m=k-5$, \Cref{thm:main-g} is equivalent to a conjecture in the author's master's thesis
\cite[Vermutung~5.2.7]{BachmannMSc}.
The terms $g_{k-4}$ and $g_{k-2}$ are the lower-weight correction to the
homogeneous weight-$k$ relation in \Cref{thm:main-G}.

For example, the first four cases of \Cref{thm:main-G} are
\begin{align*}
 G_6&=6G_{3,3}-3G_{4,2},\\
 G_7&=4G_{3,4}+3G_{4,3}-2G_{5,2},\\
 G_8&=\frac{1}{7}\bigl(20G_{3,5}+30G_{4,4}+16G_{5,3}-10G_{6,2}\bigr),\\
 G_9&=\frac{1}{14}\bigl(30G_{3,6}+60G_{4,5}+66G_{5,4}
       +35G_{6,3}-15G_{7,2}\bigr).
\end{align*}
Based on \cite{BKM}, we expect no relations among multiple Eisenstein series
in weights below $6$.  Thus, the first relation above is expected to be the
first relation among multiple Eisenstein series.

\section*{Acknowledgments}
This project was partially supported by JSPS KAKENHI Grant Number JP26K22254.

\section{Proof of the double Eisenstein series identity}

For both proofs, it is convenient to remove the factorial normalization from
$\alpha_{a,b}$.  Define
\begin{align}\label{eq:mu}
 \mu_{i,k}
 \coloneqq(2^{i-1}-1)\binom{k-6}{i-1}
 +(2^{i-2}+2)\binom{k-6}{i-2}
 -\binom{k-6}{i-3}
\end{align}
for $k\geq6$ and $2\leq i\leq k-3$.  If $a+b=k$, then
\begin{align}\label{eq:alpha-mu}
 \alpha_{a,b}
 =\frac{60}{b}\binom{k-1}{a-1}^{-1}\mu_{a-1,k}
 =\frac{60(a-1)!(b-1)!}{(k-1)!}\mu_{a-1,k}.
\end{align}
The first equality is \eqref{eq:alpha}, and the second follows by a direct
calculation.  The numbers $\mu_{i,k}$ have the following generating
polynomial:
\begin{align}\label{eq:polynomial}
 \sum_{i=2}^{k-3}\mu_{i,k}X^iY^{k-2-i}
 =XY\Bigl(Y(X+Y)(2X+Y)^{k-6}
 -(X-Y)^2(X+Y)^{k-6}\Bigr).
\end{align}

We recall the algebraic form of the restricted double-shuffle relations from \cite{BT}.  Let
\begin{align*}
 \mathfrak H=\QQ\langle x,y\rangle,
 \qquad
 \mathfrak H^1=\QQ+\mathfrak H y.
\end{align*}
For $n\geq1$, put $z_n=x^{n-1}y$, so that
$\mathfrak H^1=\QQ\langle z_1,z_2,\dots\rangle$.  The harmonic, or stuffle,
product $*$ on $\mathfrak H^1$ is the $\QQ$-bilinear product defined by
\begin{align*}
 1*w=w*1&=w,\\
 z_m u*z_n v
 &=z_m(u*z_n v)+z_n(z_m u*v)+z_{m+n}(u*v)
\end{align*}
for words $u,v,w\in\mathfrak H^1$ and $m,n\geq1$.  The shuffle product
$\shuffle$ on $\mathfrak H$ is the $\QQ$-bilinear product defined by
\begin{align*}
 1\shuffle w=w\shuffle1&=w,\\
 \varepsilon u\shuffle\delta v
 &=\varepsilon(u\shuffle\delta v)+\delta(\varepsilon u\shuffle v)
\end{align*}
for words $u,v,w\in\mathfrak H$ and letters
$\varepsilon,\delta\in\{x,y\}$.  It restricts to a product on
$\mathfrak H^1$.

Put $\mathfrak H^{\geq2}=\QQ\langle z_2,z_3,\dots\rangle$ and
$\ds(u,v)=u*v-u\shuffle v$.  The shuffle-regularized multiple Eisenstein
series of \cite{BT} define a $\QQ$-linear map
$G^\shuffle:\mathfrak H^1\to\mathcal O(\HH)$ which agrees with the lattice
sum above for $k_1\geq3$ and $k_2,\dots,k_r\geq2$.  By
\cite[Theorem~1.2]{BT}, these satisfy the restricted double-shuffle relations
\begin{align}\label{eq:rds}
 G^\shuffle\bigl(\ds(u,v)\bigr)=0
 \qquad(u,v\in\mathfrak H^{\geq2}).
\end{align}
See also \cite[Section~1]{BKM}.  The shuffle product of two depth-one words
is
\begin{align*}
 z_a\shuffle z_b
 &=\sum_{j=0}^{a-1}\binom{b-1+j}{j}z_{b+j}z_{a-j}
 +\sum_{j=0}^{b-1}\binom{a-1+j}{j}z_{a+j}z_{b-j}.
\end{align*}
We use the following restriction of \cite[Proposition~2]{GKZ} to indices
at least $2$.  For rational numbers $\rho_{r,s}$ with $r+s=k$ and
$r,s\geq2$, define
\begin{align*}
 A_\rho(X,Y)
 =\summ{r+s=k\\r,s\geq2}
 \binom{k-2}{r-1}\rho_{r,s}X^{r-1}Y^{s-1}.
\end{align*}
Suppose that
\begin{align*}
 A_\rho(X,Y)=H(X,X+Y)-H(X,Y)
\end{align*}
for a symmetric homogeneous polynomial $H\in\QQ[X,Y]$ of degree $k-2$
which is divisible by $XY$.  Then
\begin{align}\label{eq:coefficient-criterion}
 \summ{r+s=k\\r,s\geq2}\rho_{r,s}z_rz_s-\lambda z_k
 \in\operatorname{span}_{\QQ}
 \bigl\{\ds(z_a,z_{k-a}):2\leq a\leq k-2\bigr\},
 \qquad
 \lambda=\frac{k-1}{2}\int_0^1H(t,1-t)\,dt.
\end{align}
Indeed, let $P_k$ be the map from the weight-$k$ depth-two words to
$\QQ[X,Y]$ given by
\begin{align*}
 P_k(z_rz_s)=\binom{k-2}{r-1}X^{r-1}Y^{s-1}.
\end{align*}
For $a+b=k$, put
\begin{align*}
 Q_a(X,Y)=X^{a-1}Y^{b-1}+X^{b-1}Y^{a-1},
 \qquad C_a=\binom{k-2}{a-1}.
\end{align*}
The stuffle product and the shuffle formula above give
\begin{align*}
 P_k(z_a\shuffle z_b)&=C_aQ_a(X,X+Y),\\
 P_k(z_az_b+z_bz_a)&=C_aQ_a(X,Y).
\end{align*}
The polynomials $Q_a$ for $2\leq a\leq\lfloor k/2\rfloor$ form a basis
of the symmetric homogeneous polynomials of degree $k-2$ which are
divisible by $XY$.  Write
$H=\sum_{a=2}^{\lfloor k/2\rfloor}h_aQ_a$ and put $b=k-a$.  The two
formulas above show that the depth-two part of
\begin{align*}
 D=-\sum_{a=2}^{\lfloor k/2\rfloor}
 \frac{h_a}{C_a}\ds(z_a,z_b)
\end{align*}
has coefficient polynomial $A_\rho$.  Its depth-one part is
$-\sum_a(h_a/C_a)z_k$, and
\begin{align*}
 \sum_a\frac{h_a}{C_a}
 =\frac{k-1}{2}\int_0^1H(t,1-t)\,dt=\lambda,
\end{align*}
where we used
$C_a\int_0^1Q_a(t,1-t)\,dt=2/(k-1)$.  Hence $D$ is the left-hand side
of \eqref{eq:coefficient-criterion}.

\begin{lem}\label{lem:weighted-ds}
For $k\geq6$,
\begin{align}\label{eq:weighted-ds}
 \sum_{r=3}^{k-2}\alpha_{r,k-r}z_rz_{k-r}-z_k
 \in\operatorname{span}_{\QQ}
 \bigl\{\ds(z_a,z_{k-a}):2\leq a\leq k-2\bigr\}.
\end{align}
\end{lem}

\begin{proof}
Set $\rho_{r,k-r}=\alpha_{r,k-r}$ for $3\leq r\leq k-2$ and
$\rho_{2,k-2}=0$.  By \eqref{eq:alpha-mu} and \eqref{eq:polynomial}, the
corresponding polynomial in the coefficient criterion is
\begin{align*}
 A_\rho(X,Y)
 &=\frac{60}{k-1}\sum_{r=3}^{k-2}
 \mu_{r-1,k}X^{r-1}Y^{k-r-1}\\
 &=\frac{60}{k-1}XY\Bigl(
 Y(X+Y)(2X+Y)^{k-6}-(X-Y)^2(X+Y)^{k-6}\Bigr).
\end{align*}
Define the symmetric homogeneous polynomial
\begin{align*}
 H_k(X,Y)=\frac{60}{k-1}XY(X-Y)^2(X+Y)^{k-6}.
\end{align*}
Then
\begin{align*}
 A_\rho(X,Y)=H_k(X,X+Y)-H_k(X,Y).
\end{align*}
Moreover,
\begin{align*}
 \frac{k-1}{2}\int_0^1H_k(t,1-t)\,dt
 =30\int_0^1t(1-t)(2t-1)^2\,dt=1.
\end{align*}
The claim now follows from \eqref{eq:coefficient-criterion}.
\end{proof}

\begin{proof}[Proof of \Cref{thm:main-G}]
By \Cref{lem:weighted-ds} and \eqref{eq:rds}, the left-hand side of
\eqref{eq:weighted-ds} lies in the kernel of $G^\shuffle$.  All its words
have first index at least $3$, so we obtain
\begin{align*}
 G_k=\sum_{r=3}^{k-2}\alpha_{r,k-r}G_{r,k-r}.
\end{align*}
Writing $a=r$ and $b=k-r$ gives \eqref{eq:main-G}.
\end{proof}

\section{Direct proof of the divisor-sum identity}

\begin{proof}[Proof of \Cref{thm:main-g}]
This argument is independent of the multiple Eisenstein-series proof and also
displays the lower-weight terms explicitly.  Define
\begin{align*}
 A_r(q)&\coloneqq r!\,g_{r+1}(q),&
 A_{r,s}(q)&\coloneqq r!s!\,g_{r+1,s+1}(q),
\end{align*}
and their exponential generating series in the commuting variables $X,Y$
\begin{align*}
 T(X)&\coloneqq\sum_{r\geq0}A_r(q)\frac{X^r}{r!},&
 T(X,Y)&\coloneqq
 \sum_{r,s\geq0}A_{r,s}(q)\frac{X^r}{r!}\frac{Y^s}{s!}.
\end{align*}
Directly from \eqref{eq:def-g} (cf.~\cite{BK}), we get
\begin{equation}\label{eq:T}
 \begin{aligned}
 T(X)=\sum_{u,v>0}e^{vX}q^{uv},\qquad 
 T(X,Y)=
 \summ{u_1>u_2>0\\v_1,v_2>0}
 e^{v_1X+v_2Y}q^{u_1v_1+u_2v_2}.
 \end{aligned}
\end{equation}

Separating equal and unequal values of $u_1,u_2$ gives the stuffle
decomposition
\begin{align}\label{eq:stuffle-T}
 T(X)T(Y)=T(X,Y)+T(Y,X)+\Delta(X,Y),
\end{align}
where
\begin{align*}
 \Delta(X,Y)
 =\sum_{u,v_1,v_2>0}e^{v_1X+v_2Y}q^{u(v_1+v_2)}.
\end{align*}
On the other hand, summing first over $u_1,u_2$ in the product $T(X)T(Y)$
and splitting the result gives
\begin{align}\label{eq:shuffle-T}
 T(X)T(Y)=T(X+Y,X)+T(X+Y,Y)-T(X+Y)+R_0(X+Y),
\end{align}
where
\begin{align*}
 R_0(Z)=\sum_{m>0}e^{mZ}\frac{q^m}{(1-q^m)^2}.
\end{align*}

Apply $\mathcal L=\partial_X\partial_Y(\partial_X-\partial_Y)^2$ to
\eqref{eq:stuffle-T} and \eqref{eq:shuffle-T}, and then set $X=Y=Z$.
The last two terms in \eqref{eq:shuffle-T} are annihilated because they
depend only on $X+Y$.  Moreover, $\mathcal L$ is invariant under interchanging
$X$ and $Y$.  Thus, with the notation below, the two depth-two terms in
\eqref{eq:stuffle-T} give $2\mathcal B(Z)$, while those in
\eqref{eq:shuffle-T} give $2\mathcal A(Z)$.  The remaining term in
\eqref{eq:stuffle-T} gives $\mathcal D(Z)$, and hence
$2\mathcal A(Z)=2\mathcal B(Z)+\mathcal D(Z)$.  Equivalently,
\begin{align}\label{eq:A-B}
 \mathcal A(Z)-\mathcal B(Z)=\frac12\mathcal D(Z),
\end{align}
where
\begin{align*}
 \mathcal A(Z)
 &=\left.
 \partial_U\partial_V^2(\partial_U+\partial_V)T(U,V)
 \right|_{U=2Z,V=Z},\\
 \mathcal B(Z)
 &=\left.
 \partial_U\partial_V(\partial_U-\partial_V)^2T(U,V)
 \right|_{U=V=Z},\\
 \mathcal D(Z)
 &=\sum_{u,v_1,v_2>0}
 v_1v_2(v_1-v_2)^2e^{(v_1+v_2)Z}q^{u(v_1+v_2)}.
\end{align*}

We extract the coefficient of $Z^{k-6}/(k-6)!$.  On the left of
\eqref{eq:A-B}, the relevant polynomial in $v_1,v_2$ is $
 v_1v_2\Bigl(v_2(v_1+v_2)(2v_1+v_2)^{k-6}
 -(v_1-v_2)^2(v_1+v_2)^{k-6}\Bigr)$.
By \eqref{eq:polynomial}, the coefficient is
\begin{align}\label{eq:left-coeff}
 \sum_{i=2}^{k-3}\mu_{i,k}A_{i,k-2-i}(q).
\end{align}
For the right-hand side, use the elementary power-sum identity
\begin{align*}
 \sum_{v=1}^{m-1}v(m-v)(2v-m)^2
 =\frac{m(m^2-1)(m^2-4)}{30}.
\end{align*}
It follows that the coefficient of $Z^{k-6}/(k-6)!$ in
$\frac12\mathcal D(Z)$ is
\begin{align}\label{eq:right-coeff}
 \frac1{60}\bigl(4A_{k-5}(q)-5A_{k-3}(q)+A_{k-1}(q)\bigr).
\end{align}
Equating \eqref{eq:left-coeff} and \eqref{eq:right-coeff}, substituting the
definitions of $A_r,A_{r,s}$, gives
\eqref{eq:main-g} by \eqref{eq:alpha-mu}.
\end{proof}

{\bf AI \& computational resource disclosure:} The main results, their formulation, and the underlying idea of proof are the author's own. ChatGPT~5.6 and Claude Fable~5 were used throughout as research assistants: they carried out and checked computations, and a number of intermediate steps in the proofs were worked out in dialogue with these systems. All statements and proofs were verified by the author, who takes sole responsibility for them.

\end{document}